\documentclass[10pt]{amsart}

\usepackage{graphicx}
\usepackage{amsmath}
\usepackage{amssymb}

\numberwithin{equation}{section}
\def\RR{{\mathbb R}}
\def\MM{{\mathbb M}}

\def\LL{{\mathbb L}}

\def\DD{{\mathbb D}}

\def\FF{{\mathbb{F}}}
\def\Z{\mathcal{Z}}

\def\GG{{\mathbb G}}
\def\HH{{\mathbb H}}

\def\Aff{{\rm Aff}}

\def\wto{\rightharpoonup}

\def\eps{\varepsilon}

\def\ell{l}
\def\mwto{\stackrel{*}{\wto}}

\def\mint{{{\bf-}\!\!\!\!\!\!\hspace{-.1em}\int}}
\def\limsup{\mathop{\overline{\lim}}}
\def\liminf{\mathop{\underline{\lim}}}
\def\mwto{\stackrel{*}{\wto}}
\def\supp{{\rm supp}}
\def\essinf{\mathop{\rm ess\  inf}}
\def\exp{{\rm e}}

\newtheorem{theorem}{Theorem}[section]
\newtheorem{lemma}[theorem]{Lemma}
\newtheorem{proposition}[theorem]{Proposition}
\newtheorem{corollary}[theorem]{Corollary}

\theoremstyle{definition}
\newtheorem{definition}[theorem]{Definition}

\theoremstyle{remark}
\newtheorem{remark}[theorem]{Remark}

\title{Localization principle and relaxation}

\author{Jean-Philippe Mandallena}
\address{Laboratoire MIPA (Math\'ematiques, Informatique, Physique et Applications)\newline UNIVERSITE DE NIMES, Site des Carmes, Place Gabriel P\'eri, 30021 N\^\i mes, France.}
\email{jean-philippe.mandallena@unimes.fr}



\begin{document}

\begin{abstract}
Relaxation theorems for multiple integrals on $W^{1,p}(\Omega;\RR^m)$, where $p\in]1,\infty[$, are proved under general conditions on the integrand $L:\MM\to[0,\infty]$ which is Borel measurable and not necessarily finite. We involve a localization principle that we previously used to prove a general lower semicontinuity result. We apply these general results to the relaxation of nonconvex integrals with exponential-growth.
\end{abstract}

\maketitle


\section{Introduction}

This paper is concerned with the problem of finding an integral representation for the ``relaxed functional" $\overline{I}:W^{1,p}(\Omega;\RR^m)\to[0,\infty]$, with $p\in]1,\infty[$, given by
$$
\overline{I}(u)=\inf\left\{\liminf_{n\to\infty}\int_\Omega L(\nabla u_n(x))dx:u_n\to u\hbox{ in }L^p(\Omega;\RR^m)\right\}
$$
with $\Omega\subset\RR^N$ a bounded open set with Lipschitz boundary and $L:\MM\to[0,\infty]$  a Borel measurable and $p$-coercive function, i.e., $L(\cdot)\geq C|\cdot|^p$ for some $C>0$, where $\MM$ denotes the space of all real $m\times N$ matrices.

The concept of ``relaxed functional" was introduced by Bogolubov in 1930 in the case where $m=N=1$ (see \cite{bogolubov30}). Bogolubov developed a one-dimensional theory of relaxation and highlighted the key role of the convexification of $L$ for representing $\overline{I}$. Later, in 1974, Ekeland and Temam extended Bogolubov's relaxation theory to the case $N\geq 1$ and $m=1$ (see \cite{ekeland-temam74}, see also \cite{ioffe-tihomirov68}). Then, in the vector case, i.e., $\min\{m,N\}>1$, in 1982, Dacorogna in \cite{dacorogna82} (see also \cite{dacorogna08}) proved that if $L$ has $p$-growth, i.e., $L(\cdot)\leq c(1+|\cdot|^p)$ for some $c>0$, then $\overline{I}$ can be represented by an integral whose integrand is given by the $W^{1,\infty}$-quasiconvexication, but not the convexification, of $L$, i.e., $\Z_\infty L:\MM\to[0,\infty]$ defined by
$$
\Z_\infty L(\xi):=\inf\left\{\int_YL(\nabla\phi(y))dy:\phi\in \l_\xi+W^{1,\infty}_0(Y;\RR^m)\right\}
$$
with $Y:=]-{1\over 2},{1\over 2}[^N$ and $\l_\xi(y):=\xi y$. The paper of Dacorogna, i.e., \cite{dacorogna82}, is the point of departure of many works on the problem of representing $\overline{I}$ by a variational integral in the vector case: several authors have tried to extend Dacorogna's theorem to more general classes of integrands (see for instance \cite{acerbi-fusco84,ball-murat84,marcellini86,fonseca-maly97,benbelgacem00,sychev05a,oah-jpm07,oah-jpm08a,oah10,sychev10,sychev11,sychev11-2}, see also for concrete examples \cite{adamscontidesimonedolzmann08,contidolzmannklust09}). In the same spirit, the object of the present paper is to study the existence of an integral representation of $\overline{I}$ with respect to a general condition on $L$, that we called ``localization principle", i.e., 
\begin{itemize}
\item[(C$_{p,q}$)] for every $\xi\in\MM$ and every $\{v_n\}_n\subset W^{1,p}(Y;\RR^m)$ such that
$$
\left\{
\begin{array}{l}
v_n\to l_\xi\hbox{ in }L^p(Y;\RR^m)\\
\displaystyle\sup_n\int_YL(\nabla v_n(y))dy<\infty,
\end{array}
\right.
$$
there exist a subsequence $\{v_n\}_n$ (not relabeled) and a sequence $\{w_n\}_n\subset l_\xi+W^{1,q}_0(Y;\RR^m)$ such that
$$
\left\{
\begin{array}{l}
\displaystyle\left|\nabla v_n-\nabla w_n\right|\to0\hbox{ in measure}\\
\displaystyle\left\{L(\nabla w_n)\right\}_n\hbox{ is equi-integrable.,}
\end{array}
\right.
$$
\end{itemize}
that we used in \cite{jpm11} to prove the following lower semicontinuity theorem.
\begin{theorem}\label{lsc-theorem}
Consider $p\in]1,\infty[$ and $q\in[1,\infty]$ and assume that $L$ is continuous. If {\rm (C$_{p,q}$)} holds and if $L$ is $W^{1,q}$-quasiconvex, i.e., $L=\Z_qL$ with $\Z_qL:\MM\to[0,\infty]$ given by
$$
\Z_q L(\xi):=\inf\left\{\int_YL(\nabla\phi(y))dy:\phi\in l_\xi+W^{1,q}_0(Y;\RR^m)\right\},
$$
then the functional $I:W^{1,p}(\Omega;\RR^m)\to[0,\infty]$ defined by
$$
I(u):=\int_\Omega L(\nabla u(x))dx
$$
is lower semicontinuous with respect to the strong convergence of $L^p(\Omega;\RR^m)$.
\end{theorem}
Note that the condition (C$_{p,q}$) is a generalization of previous ``localization principles" (see \cite{acerbi-fusco84,K-P92,sychev99,sychev05a}). Our goal here is to study how (C$_{p,q}$) behave with respect to the relaxation.

In this paper we establish two extensions of Theorem \ref{lsc-theorem} to the case of the relaxation (see Theorems \ref{Main-Theorem} and \ref{Main-Theorem-Bis}). Roughly, the first theorem, i.e., Theorem \ref{Main-Theorem}, asserts that, when the effective domain $\LL$ of $L$ is closed, $L$ is continuous on $\LL$ and $q\geq p$, if (C$_{p,q}$) holds and if $W^{1,p}$-sobolev functions can be approximated by piecewise affine functions both in $L^p$-norm and in $\Z_qL$-energy, see (H$_{p,q}$) in the next section, then  $\overline{I}$ has an integral representation whose integrand is given the $W^{1,q}$-quasiconvexification of $L$, i.e., by $\Z_qL$. The second theorem, i.e., Theorem \ref{Main-Theorem-Bis}, which is a variant of the first one, says that if $L$ is continuous on the interior of $\LL$ and if $q\geq p$, then under additional assumptions related to the behavior of $L$ on the boundary of $\LL$, see (R$_1$), (R$_2$) and (R$_3$) in the next section, and under slight modifications of (C$_{p,q}$) and (H$_{p,q}$), see ($\widehat {\rm C}_{p,q}$) and ($\widehat {\rm H}_{p,q}$) in the next section, the functional $\overline{I}$ can be represented by an integral whose integrand is given by the lower semicontinuous envelope of $\Z_qL$ (see Remark \ref{Remark-lsc-ZqL}). Note that, in our context, to obtain the full integral representations of $\overline{I}$ in Theorems \ref{Main-Theorem} and \ref{Main-Theorem-Bis}, it seems difficult to avoid assumption (H$_{p,q}$) or its variant ($\widehat {\rm H}_{p,q}$) (such a fact was also pointed out in \cite{sychev11}). However, partial integral representations of $\overline{I}$ on piecewise affine functions can be established under only (C$_{p,q}$) or its variant ($\widehat{\rm C}_{p,q}$) together with (R$_1$), (R$_2$) and (R$_3$) (see Theorems \ref{Main-Theorem}(a) and \ref{Main-Theorem-Bis}(a)).

The plan of the paper is as follows. In \S 2 we state the main results, i.e., Theorems \ref{Main-Theorem} and \ref{Main-Theorem-Bis}, that we prove respectively in \S 4 and \S 5. The proofs of Theorems \ref{Main-Theorem}(a) and \ref{Main-Theorem-Bis}(a) use in a fundamental way Young measure theory whose brief summary are given in \S 3.1. The proof of Theorem \ref{Main-Theorem-Bis} also need the concept of radially uniformly upper semicontinuous integrand, recently introduced in \cite{oah10}, whose definition together with some facts are recalled in \S3.2. Finally, in \S 6 we apply our two main theorems to the relaxation of nonconvex integrals with exponential-growth (see Corollary \ref{Applications-Of-Main-Theorems}).
 
\section{Main results}

Let $m,N\geq 1$ be two integers, let $\Omega\subset\RR^N$ be a bounded open set with Lipschitz boundary and let $\MM$ denote the space of all real $m\times N$ matrices. Let $p\in]1,\infty[$, let $L:\MM\to[0,\infty]$ be a Borel measurable function which is $p$-coercive, i.e., $L(\cdot)\geq C|\cdot|^p$ for some $C>0$, and let $I:W^{1,p}(\Omega;\RR^m)\to[0,\infty]$ be defined by
$$
I(u):=\int_\Omega L(\nabla u(x))dx.
$$
Denote the lower semicontinuous envelope (or the ``relaxed functional") of $I$ with respect to the strong convergence in $L^p(\Omega;\RR^m)$ by $\overline{I}:W^{1,p}(\Omega;\RR^m)\to[0,\infty]$, i.e.,
$$
\overline{I}(u):=\inf\left\{\liminf_{n\to\infty}I(u_n):u_n\to u\hbox{ in }L^p(\Omega;\RR^m)\right\},
$$
and for each $q\in]1,\infty]$, define $\Z_qL:\MM\to[0,\infty]$ by
$$
\Z_qL(\xi):=\inf\left\{\int_YL(\nabla \phi(y))dy:\phi\in l_\xi+W^{1,q}_0(Y;\RR^m)\right\},
$$
where $Y:=]-{1\over 2},{1\over 2}[^N$ and $l_\xi(y):=\xi y$. Usually, $\Z_qL$ is called the $W^{1,q}$-quasiconvexi-fication of $L$.  Given $p\in]1,\infty[$ and $q\in]1,\infty]$, we consider two general conditions on $L$, i.e., (C$_{p,q}$) stated in the introduction  and (H$_{p,q}$) below.
 (In what follows, $\Aff(\Omega;\RR^m)$ denotes the space of continuous piecewise affine functions from $\Omega$ to $\RR^m$.)
\begin{itemize}
\item[(H$_{p,q}$)] For every $u\in W^{1,p}(\Omega;\RR^m)\setminus\Aff(\Omega;\RR^m)$ such that
$$
\int_\Omega\Z_qL(\nabla u(x))dx<\infty,
$$
there exists $\{u_n\}_n\subset\Aff(\Omega;\RR^m)$ such that
$$
\left\{
\begin{array}{l}
u_n\to u\hbox{ in }L^p(\Omega;\RR^m)\\
\displaystyle\limsup_{n\to\infty}\int_\Omega\Z_qL(\nabla u_n(x))dx\leq\int_\Omega\Z_qL(\nabla u(x))dx.
\end{array}
\right.
$$
\end{itemize}

Denote the effective domain of $L$ by $\LL$. The first main result of the paper is the following. 

\begin{theorem}\label{Main-Theorem}
Consider $p\in]1,\infty[$ and $q\in[p,\infty]$ and assume that $\LL$ is closed and $L$ is continuous on $\LL$.
\begin{itemize}
\item[(a)] If  {\rm (C$_{p,q}$)} holds then for every $u\in \Aff(\Omega;\RR^m)$
\begin{equation}\label{Eq-Th-1}
\overline{I}(u)=\int_\Omega\Z_qL(\nabla u(x))dx
\end{equation}
\item[(b)] If  {\rm (C$_{p,q}$)} and {\rm (H$_{p,q}$)} are satisfied then \eqref{Eq-Th-1} holds for all $u\in W^{1,p}(\Omega;\RR^m)$.
\end{itemize}
\end{theorem}

In order to take more general situations into account like singular behavior of $L$ on $\partial\LL$ (see \cite{oah10,oah-jpm10b,oah-jpm10a} and also \cite{carbone-dearcangelis02} for the scalar case) we are led to replace  the assumption ``$\LL$ is closed and $L$ is continuous on $\LL$" by the weaker one ``$L$ is continuous on the interior of $\LL$". In our context, this can be done by considering the three additional conditons (R$_1$), (R$_2$) and (R$_3$) below together with slight modifications of {\rm (C$_{p,q}$)} and {\rm (H$_{p,q}$)}, i.e., {\rm ($\widehat {\rm C}_{p,q}$)} and {\rm ($\widehat {\rm H}_{p,q}$)} below.

\begin{itemize}
\item[(R$_1$)] $L$ is radially uniformly upper semicontinuous (ru-usc), i.e., there exists $c>0$ such that
$$
\limsup_{t\to1}\Delta^c_{L}(t)\leq0
$$ 
with $\Delta^c_{L}:[0,1]\to]-\infty,\infty]$ given by
\begin{equation}\label{Ru-usc-Z-hat-i}
\Delta^c_{L}(t):=\sup_{\xi\in\LL}{L(t\xi)-L(\xi)\over c+L(\xi)}.
\end{equation}
\end{itemize}
(The concept of ru-usc integrand was introduced by Anza Hafsa in \cite{oah10}, see also \cite{oah-jpm10b,oah-jpm10a}.) Note that if $L$ is ru-usc, also is $\Z_qL$ (see Proposition \ref{Coro-Prop-ru-usc1}). In what follows, given any $\DD\subset \MM$, ${\rm int}(\DD)$ (resp. $\overline{\DD}$) denotes the interior (resp. the closure) of $\DD$.

\begin{itemize}
\item[(R$_2$)] $t\overline{\LL}\subset{\rm int}(\LL)$ for all $t\in]0,1[$.
\end{itemize}
Denote the effective domain of $\Z_qL$ by $\Z_q\LL$.
\begin{itemize} 
\item[(R$_3$)] $t\overline{\Z_q\LL}\subset{\rm int}(\Z_q\LL)$ for all $t\in]0,1[$. 
\end{itemize}
The following condition is a variant of (C$_{p,q}$).
\begin{itemize}
\item[($\widehat{\rm C}_{p,q}$)] For every $t\in]0,1[$, every $\xi\in\MM$ and every $\{v_n\}_n\subset W^{1,p}(Y;\RR^m)$ such that
$$
\left\{
\begin{array}{l}
v_n\to l_\xi\hbox{ in }L^p(Y;\RR^m)\\
\displaystyle\sup_n\int_YL(\nabla v_n(y))dy<\infty,\\
\end{array}
\right.
$$
there exist a subsequence $\{v_n\}_n$ (not relabeled) and a sequence $\{w_n\}_n\subset l_{t\xi}+W^{1,q}_0(Y;\RR^m)$ and there is $s\in]0,1[$ such that
$$
\left\{
\begin{array}{l}
\displaystyle\left|t\nabla v_n-\nabla w_n\right|\to0\hbox{ in measure}\\
\displaystyle\left\{L(\nabla w_n)\right\}_n\hbox{ is equi-integrable}\\
\nabla w_n(y)\in s\overline{\LL}\hbox{ for all }n\geq 1\hbox{ and a.a. }y\in Y.
\end{array}
\right.
$$
\end{itemize}
The following condition is a variant of (H$_{p,q}$).
\begin{itemize}
\item[($\widehat{\rm H}_{p,q}$)] For every $t\in]0,1[$ and every $u\in W^{1,p}(\Omega;\RR^m)\setminus\Aff(\Omega;\RR^m)$ such that
$$
\left\{
\begin{array}{l}
\displaystyle\int_\Omega\Z_qL(\nabla u(x))dx<\infty\\
\nabla u(x)\in t\overline{\Z_q\LL}\hbox{ for a.a. }x\in \Omega,
\end{array}
\right.
$$
there exists $\{u_n\}_n\subset\Aff(\Omega;\RR^m)$ such that
$$
\left\{
\begin{array}{l}
u_n\to u\hbox{ in }L^p(\Omega;\RR^m)\\
\displaystyle\limsup_{n\to\infty}\int_\Omega\Z_qL(\nabla u_n(x))dx\leq\int_\Omega\Z_qL(\nabla u(x))dx.
\end{array}
\right.
$$
\end{itemize}
The second main result of the paper is the following. 

\begin{theorem}\label{Main-Theorem-Bis}
Consider $p\in]1,\infty[$ and $q\in[p,\infty]$ and assume that $L$ is continuous on ${\rm int}(\LL)$ and {\rm(R$_1$)}, {\rm(R$_2$)} and {\rm(R$_3$)} are verified.
\begin{itemize}
\item[(a)] If {\rm ($\widehat{\rm C}_{p,q}$)} holds then for every $u\in\Aff(\Omega;\RR^m)$
\begin{equation}\label{Eq-Th2}
\overline{I}(u)=\int_\Omega\widehat{\Z_qL}(\nabla u(x))dx
\end{equation}
with $\widehat{\Z_qL}:\MM\to[0,\infty]$ given by 
\begin{equation}\label{Def-Z-hat-i}
\widehat{\Z_qL}(\xi):=\liminf_{t\to1}\Z_qL(t\xi).
\end{equation}
\item[(b)] If {\rm ($\widehat{\rm C}_{p,q}$)} and {\rm ($\widehat{\rm H}_{p,q}$)} are satisfied then \eqref{Eq-Th2} holds for all $u\in W^{1,p}(\Omega;\RR^m)$.
\end{itemize}
\end{theorem}

\begin{remark}\label{Remark-lsc-ZqL}
According to Lemma \ref{Fonseca-Lemma} and Theorem \ref{Extension-Result-for-ru-usc-Functions}, we see that, in Theorem \ref{Main-Theorem-Bis}, the function $\widehat{\Z_qL}$ is in fact the lower semicontinuous envelope of $\Z_qL$ and has the representation
$$
\widehat{\Z_qL}(\xi)=\left\{
\begin{array}{ll}
\Z_qL(\xi)&\hbox{if }\xi\in{\rm int}(\Z_q\LL)\\
\lim\limits_{t\to 1}\Z_qL(t\xi)&\hbox{if }\xi\in\partial(\Z_q\LL)\\
\infty&\hbox{otherwise.}
\end{array}
\right.
$$
\end{remark}


\section{Auxiliary results}

\subsection{Some facts on Young measures}

Young measures were introduced by Young in 1937 (see \cite{young37}) with the purpose of finding an extension of the class of Sobolev functions for which one-dimensional nonconvex variational problems become solvable. In the context of the multidimensional calculus of variations, Kinderleherer and Pedregal (see \cite{K-P92,K-P94}) and independently Kristensen (see \cite{kristensen94}) were the first to use Young measures for dealing with lower semicontinuity problems. Relaxation and convergence in energy problems were studied for the first time by Sychev via Young measures following a new approach to Young measures that he introduced in \cite{sychev99}. Here we only recall the ingredients that we need for proving Theorems \ref{Main-Theorem} and \ref{Main-Theorem-Bis}. For more details on Young measure theory and its applications to the calculus of variations we refer to \cite{pedregal97,pedregal00,sychev04b}.

Let $\mathcal{P}(\MM)$ be the  set of all probability measures on $\MM$, let $C(\MM)$ be the space of all continuous functions from $\MM$ to $\RR$ and let
$$
C_0(\MM):=\Big\{\Phi\in C(\MM):\lim_{|\xi|\to\infty}\Phi(\xi)=0\Big\}.
$$ 

Here is the definition of a Young measure.

\begin{definition}\label{Def-of-Young-Measures}
A family $(\mu_x)_{x\in\Omega}$ of probability measures on $\MM$, i.e., $\mu_x\in\mathcal{P}(\MM)$ for all $x\in\Omega$, is said to be a Young measure if there exists a sequence $\{\xi_n\}_n$ of measurable functions from $\Omega$ to $\MM$ such that
$$
\Phi(\xi_n)\mwto\langle\Phi;\mu_{(\cdot)}\rangle\hbox{ in }L^\infty(\Omega)\hbox{ for all }\Phi\in C_0(\MM)
$$
with $\langle\Phi;\mu_{(\cdot)}\rangle:=\int_{\MM}\Phi(\zeta)d\mu_{(\cdot)}(\zeta)$. In this case, we say that $\{\xi_n\}_n$ generates $(\mu_x)_{x\in\Omega}$ as a Young measure.
\end{definition}

The following lemma makes clear the link between convergence in measure and Young measures. (The proof follows from the definition.)

\begin{lemma}\label{Lemma1-YM}
let $\{\xi_n\}_n$ and $\{\zeta_n\}_n$ be two sequences of measurable functions from $\Omega$ to $\MM$. If $\{\xi_n\}_n$ generates a Young measure and if $|\xi_n-\zeta_n|\to0$ in measure then $\{\zeta_n\}_n$ generates the same Young measure.
\end{lemma}

The following theorem gives a sufficient condition for proving the existence of Young measures (for a proof see \cite{ball89,sychev04b,fonseca-leoni07}).

\begin{theorem}\label{ExiStenCE-Young-measures}
Let $\theta:\MM\to\RR$ be a continuous function such that 
$
\lim_{|\zeta|\to\infty}\theta(\zeta)=\infty
$
and let $\{\xi_n\}_n$ be a sequence of measurable functions from $\Omega$ to $\MM$ such that
$$
\sup_{n}\int_\Omega \theta(\xi_n(x))dx<\infty.
$$
Then,  $\{\xi_n\}_n$ contains a subsequence generating a Young measure.
\end{theorem}

The following two theorems are important in dealing with integral functionals (for proofs see \cite{balder84,sychev99}).

\begin{theorem}[semicontinuity theorem]\label{S-T-YM}
Let $G:\MM\to[0,\infty]$ be a Borel measurable function which is continuous on a closed set $\DD\subset\MM$ and let $\{\xi_n\}_n$ be a sequence of measurable functions from $\Omega$ to $\MM$ such that 
$$
\left\{
\begin{array}{l}
\xi_n(x)\in\DD\hbox{ for all }n\geq 1\hbox{ and a.a. }x\in\Omega;\\
\{\xi_n\}_n\hbox{ generates }(\mu_x)_{x\in\Omega}\hbox{ as a Young measure.} 
\end{array}  
\right.
$$
Then 
$$
\liminf_{n\to\infty}\int_\Omega G(\xi_n(x))dx\geq\int_\Omega \langle G;\mu_x\rangle dx.
$$
\end{theorem}

\begin{theorem}[continuity theorem]\label{C-T-YM}
Let $G:\MM\to[0,\infty]$ be a Borel measurable function which is continuous on a closed set $\DD\subset\MM$ and let $\{\xi_n\}_n$ be a sequence of measurable functions from $\Omega$ to $\MM$ such that 
$$
\left\{
\begin{array}{l}
\xi_n(x)\in\DD\hbox{ for all }n\geq 1\hbox{ and a.a. }x\in\Omega;\\
\{\xi_n\}_n\hbox{ generates }(\mu_x)_{x\in\Omega}\hbox{ as a Young measure.} 
\end{array}  
\right.
$$
Then 
$$
\lim\limits_{n\to\infty}\int_\Omega G(\xi_n(x))dx=\int_\Omega\langle G;\mu_x\rangle dx<\infty
$$
if and only if $\{G(\xi_n)\}_n$ is equi-integrable.
\end{theorem}

\subsection{Some facts on ru-usc integrands}

The concept of ru-usc integrand was introduced by Anza Hafsa in \cite{oah10} to deal with relaxation of variational integrals in $W^{1,\infty}$ with constraints on the gradient in the framework of the multidimensional calculus of variations (see also \cite{wagner09}). In fact, in the scalar case, such constrained relaxation and homogenization problems was previously intensively studied by Carbone and De Arcangelis (see \cite{carbone-dearcangelis02}), but their techniques could not be generalized to the vector case. Recently, developing  the concept of ru-usc integrand, we succeeded to deal with constraints on the gradient in the context of homogenization of multiple integrals in $W^{1,p}$, with $p\in]1,\infty]$, in the vector case (see \cite{oah-jpm10b,oah-jpm10a}). Here we only recall the ingredients that we need for proving Theorem \ref{Main-Theorem-Bis}. For more details on the concept of ru-usc integrand we refer to \cite[\S 3.1]{oah-jpm10b}.

Let $G:\MM\to[0,\infty]$ be a Borel measurable function whose effective domain is denoted by $\GG$. For each $c>0$ we define $\Delta_G^c:[0,1]\to]-\infty,\infty]$ by
$$
\Delta_G^c(t):=\sup_{\xi\in \GG}{G(t\xi)-G(\xi)\over c+G(\xi)}.
$$

\begin{definition}\label{ru-usc-Def}
We say that $G$ is radially uniformly upper semicontinuous (ru-usc) if there exists $c>0$ such that
$
\limsup_{t\to 1}\Delta^c_G(t)\leq 0.
$
\end{definition}

Define $\widehat{G}:\MM\to[0,\infty]$ by 
$$
\widehat{G}(\xi):=
\liminf_{t\to 1}G(t\xi).
$$

The interest of Definition \ref{ru-usc-Def} comes from the following theorem (for a proof see \cite[\S 3.1]{oah-jpm10b}). 

\begin{theorem}\label{Extension-Result-for-ru-usc-Functions}
If $G$ is ru-usc and 
\begin{equation}\label{Homothecie-Assumption-Bis}
t\overline{\GG}\subset{\rm int}(\GG)\hbox{ for all }t\in]0,1[,
\end{equation}
where $\overline{\GG}$ (resp. ${\rm int}(\GG)$) denotes the closure (resp. the interior) of $\GG$, and $G$ is lower semicontinuous (lsc) on ${\rm int}(\GG)$ then{\rm:}
\begin{itemize}
\item[(i)] $
\widehat{G}(\xi)=\left\{
\begin{array}{ll}
G(\xi)&\hbox{if }\xi\in{\rm int}(\GG)\\
\lim\limits_{t\to 1}G(t\xi)&\hbox{if }\xi\in\partial\GG\\
\infty&\hbox{otherwise{\rm;}}
\end{array}
\right.
$
\item[(ii)] $\widehat{G}$ is ru-usc{\rm;}
\item[(iii)] $\widehat{G}$ is the lsc envelope of $G$.
\end{itemize}
\end{theorem}

Let $q\in]1,\infty]$ and let $\Z_qG:\MM\to[0,\infty]$ be given by
\begin{equation}\label{ZqG-Def}
\Z_q G(\xi):=\inf\left\{\int_Y G(\nabla\phi(y))dy:\phi\in l_\xi+W^{1,q}_0(Y;\RR^m)\right\}
\end{equation}
with $Y:=]-{1\over 2},{1\over 2}[^N$ and $l_\xi(y):=\xi y$. The following proposition shows that ru-usc integrands have a nice behavior with respect to the $W^{1,q}$-quasiconvexification (for a proof see \cite[\S 3.1]{oah-jpm10b}).

\begin{proposition}\label{Coro-Prop-ru-usc1}
If $G$ is ru-usc then $\Z_q G$ is ru-usc.
\end{proposition}

\subsection{Two properties of the $W^{1,q}$-quasiconvexification formula} Here, we give two (classical) properties of the $W^{1,q}$-quasiconvexification formula, i.e., $\Z_qG:\MM\to[0,\infty]$ defined by \eqref{ZqG-Def}. Denote the effective domain of $\Z_qG$ by $\Z_q\GG$. Lemma \ref{Fonseca-Lemma} below is due to Fonseca (see \cite{fonseca88}). 
\begin{lemma}\label{Fonseca-Lemma}
$\Z_qG$ is continuous on ${\rm int}(\Z_q\GG)$.
\end{lemma}
The proof of the following lemma can be found in \cite[\S 3.4]{oah-jpm10b} (see also \cite{oah-jpm08a,oah-jpm07}).
\begin{lemma}\label{Localization-Lemma-ZG}
Given $\xi\in\MM$ and a bounded open set $U\subset\RR^N$ there exists $\{\phi_n\}_n\subset W^{1,q}_0(U;\RR^m)$ such that
$$
\left\{
\begin{array}{l}
\displaystyle\lim_{n\to\infty}\|\phi_n\|_{L^q(U;\RR^m)}=0\\
\displaystyle\lim_{n\to\infty}\mint_UG(\xi+\nabla\phi_n(x))dx=\Z_qG(\xi).
\end{array}
\right.
$$
\end{lemma}

\subsection{Approximation of integrals with convex-growth} Here we assume that $L$ has $H$-convex-growth, i.e., $\beta H(\cdot)\leq G(\cdot)\leq\alpha(1+H(\cdot))$ for some $\alpha,\beta>0$ and some convex function $H:\MM\to[0,\infty]$. Then, it is easy to see that the effective domain of $L$ is equal to the effective domain of $H$, which is convex, denoted by $\HH$ and assumed to contain the zero matrix in its interior, i.e., $0\in{\rm int}(\HH)$. We also suppose that the bounded open set $\Omega\subset\RR^N$ is strongly star-shaped, i.e., there exists $x_0\in\Omega$ such that $\overline{-x_0+\Omega}\subset t(-x_0+\Omega)$ for all $t>1$. The proof of Lemma \ref{ApproX-Lemma-1} below can be found in \cite[\S 3.3]{oah-jpm10b}.

\begin{lemma}\label{ApproX-Lemma-1}
Let $p\in]1,\infty[$ and let $u\in W^{1,p}(\Omega;\RR^m)$ be such that
$$
\left\{
\begin{array}{l}
\displaystyle\int_\Omega G(\nabla u(x))dx<\infty\\
\nabla u(x)\in {\rm int}(\HH)\hbox{ for a.a. }x\in \Omega.
\end{array}
\right.
$$
If $G$ is continuous on {\rm int$(\HH)$} then there exists $\{u_n\}_n\subset\Aff(\Omega;\RR^m)$ such that
$$
\left\{
\begin{array}{l}
u_n\to u\hbox{ in }W^{1,p}(\Omega;\RR^m)\\
\displaystyle\lim_{n\to\infty}\int_\Omega\Z_p G(\nabla u_n(x))=\int_\Omega\Z_p G(\nabla u(x))dx.
\end{array}
\right.
$$
\end{lemma}


\section{Proof of Theorem \ref{Main-Theorem}}

In this section we prove Theorem \ref{Main-Theorem}

\subsection{Proof of Theorem \ref{Main-Theorem}(a)} We are going to prove the following two inequalities: 
\begin{eqnarray}
&&\overline{I}(u)\geq\int_\Omega\Z_qL(\nabla u(x))dx\hbox{ for all }u\in W^{1,p}(\Omega;\RR^m);\label{FI-MTh-1}\\
&&\overline{I}(u)\leq\int_\Omega\Z_qL(\nabla u(x))dx\hbox{ for all }u\in \Aff(\Omega;\RR^m).\label{FI-MTh-2}
\end{eqnarray}

\begin{proof}[\bf Proof of (\ref{FI-MTh-1})]
Consider $u\in W^{1,p}(\Omega;\RR^m)$  and $\{u_n\}_n\subset W^{1,p}(\Omega;\RR^m)$ such that 
\begin{equation}\label{P-MT-Eq1}
\|u_n-u\|_{L^p(\Omega;\RR^m)}\to0,
\end{equation}
and prove that
\begin{equation}\label{P-MT-Eq0}
\liminf_{n\to\infty}I(u_n)\geq \int_\Omega\Z_qL(\nabla u(x))dx.
\end{equation}
\subsection*{Step 1: localization} Without loss of generality we can assume that
\begin{equation}\label{P-MT-Eq2}
\infty>\liminf_{n\to\infty}I(u_n)=\lim_{n\to\infty}I(u_n) \hbox{ and so }\sup_n\int_\Omega L(\nabla u_n(x))dx<\infty.
\end{equation}
Hence $\nabla u_n(x)\in\LL$ for all $n\geq 1$ and a.a. $x\in\Omega$,where $\LL$ denotes the effective domain of $L$. Since $L$ is $p$-coercive, from \eqref{P-MT-Eq2} we see that $\sup_n\int_\Omega|\nabla u_n(x)|^pdx<\infty$ and so, by Theorem \ref{ExiStenCE-Young-measures}, there exists a family $(\mu_x)_{x\in\Omega}$ of probability measures on $\MM$ such that (up to a subsequence)
\begin{equation}\label{P-MT-Eq3}
\{\nabla u_n\}_n\hbox{ generates }(\mu_x)_{x\in\Omega}\hbox{ as a Young measure}.
\end{equation}
As $L$ is continuous on $\LL$ and $\LL$ is closed, from Theorem \ref{S-T-YM} it follows that
$$
\liminf_{n\to\infty}I(u_n)\geq \int_{\Omega}\langle  L;\mu_x\rangle dx
$$
with (because \eqref{P-MT-Eq2} holds) for a.e. $x_0\in\Omega$,
\begin{equation}\label{P-MT-Eq5}
\langle L;\mu_{x_0}\rangle<\infty.
\end{equation}
Thus, to prove \eqref{P-MT-Eq0} it is sufficient to show that for a.e. $x_0\in\Omega$,
\begin{equation}\label{P-MT-Eq5-goal}
\langle L;\mu_{x_0}\rangle\geq \Z_qL(\nabla u(x_0)).
\end{equation}

\subsection*{Step 2: blow up} From \eqref{P-MT-Eq2} we deduce that there exist $f\in L^1(\Omega;[0,\infty[)$ and a finite positive Radon measure $\lambda$ on $\Omega$ with $|\supp(\lambda)|=0$ such that (up to a subsequence) $L(\nabla u_n)dx\mwto fdx+\lambda$ in the sense of measures and for a.e. $x_0\in\Omega$,
\begin{equation}\label{P-MT-Eq6}
\lim_{r\to0}\lim_{n\to\infty}\mint_{x_0+rY}L(\nabla u_n(x))dx=f(x_0)<\infty
\end{equation}
with $Y:=]-{1\over 2},{1\over 2}[^N$. As $u\in W^{1,p}(\Omega;\RR^m)$ it follows that $u$ is a.e. $L^p$-differentiable (see \cite[Theorem 3.4.2 p.129]{ziemer89}), i.e., for a.e. $x_0\in\Omega$,
\begin{equation}\label{P-MT-Eq7}
\lim_{r\to0}{1\over r^{N+p}}\big\|u(x_0+\cdot)-u(x_0)-\nabla u(x_0)y\big\|^p_{L^p({r} Y;\RR^m)}=0.
\end{equation}
From \eqref{P-MT-Eq1} we see that (up to a subsequence) for a.e. $x_0\in\Omega$,
\begin{equation}\label{P-MT-Eq8}
|u_n(x_0)-u(x_0)|^p\to0.
\end{equation}
As $C_0(\MM)$ is separable we can assert that for a.e. $x_0\in\Omega$, $x_0$ is a Lebesgue point of $\langle\Phi;\mu_{(\cdot)}\rangle$ for all $\Phi\in C_0(\MM)$, i.e., 
\begin{equation}\label{P-MT-Eq9}
\lim_{r\to 0}\mint_{x_0+rY}\langle\Phi,\mu_x\rangle dx=\langle\Phi,\mu_{x_0}\rangle\hbox{ for all }\Phi\in C_0(\MM).
\end{equation} 
Fix any $x_0\in\Omega$  such that \eqref{P-MT-Eq5}, \eqref{P-MT-Eq6}, \eqref{P-MT-Eq7}, \eqref{P-MT-Eq8} and \eqref{P-MT-Eq9} hold and fix $r_0>0$ such that $x_0+rY\subset\Omega$ for all $r\in]0,r_0]$. For each $n\geq 1$ and each $r\in]0,r_0]$, let $u_n^r\in W^{1,p}(Y;\RR^m)$ and a family  $(\mu_{y}^r)_{y\in Y}$ of probability measures  on $\MM$ be given by
$$
\left\{
\begin{array}{ll}
u_n^r(y):={1\over r}\left(u_n(x_0+ry)-u_n(x_0)\right)\\
\mu_y^r:=\mu_{x_0+ry}.
\end{array}
\right.
$$
Then \eqref{P-MT-Eq6} can be rewritten as
\begin{equation}\label{P-MT-Eq6-Bis}
\lim_{r\to0}\lim_{n\to\infty}\int_{Y}L(\nabla u_n^r(x))dx<\infty.
\end{equation}
Taking \eqref{P-MT-Eq3} into account it is easy to see that for every $r\in]0,r_0]$, $\{\nabla u_n^r\}_n$ generates $(\mu_y^r)_{y\in Y}$ as a Young measure, i.e.,
\begin{equation}\label{P-MT-Eq10}
\Phi(\nabla u_n^r)\mwto\langle\Phi,\mu_{(\cdot)}^r\rangle\hbox{ in }L^\infty(Y)\hbox{ as }n\to\infty\hbox{ for all }\Phi\in C_0(\MM),
\end{equation}
and using \eqref{P-MT-Eq9} it is clear that
\begin{equation}\label{P-MT-Eq11}
\langle\Phi;\mu_{(\cdot)}^r\rangle\mwto\langle\Phi;\mu_{x_0}\rangle\hbox{ in }L^\infty(Y)\hbox{ as }r\to0\hbox{ for all }\Phi\in C_0(\MM).
\end{equation} 
On the other hand, we have
\begin{eqnarray*}
\|u_{n}^{r}-l_{\nabla u(x_0)}\|^p_{L^p(Y;\RR^m)}&=&\int_{Y}|u_{n}^{r}(y)-l_{\nabla u(x_0)}(y)|^pdy\\
&=&{1\over r^{N+p}}\|u_n(x_0+\cdot)-u_n(x_0)-l_{\nabla u(x_0)}\|^p_{L^p(rY;\RR^m)},
\end{eqnarray*}
and consequently
\begin{eqnarray*}
\|u_{n}^r-l_{\nabla u(x_0)}\|^p_{L^p(Y;\RR^m)}&\leq&{c\over r^{N+p}}\|u_n-u\|^p_{L^p(\Omega;\RR^m)}+{c\over r^{N+p}}|u_n(x_0)-u(x_0)|^p\\
&&+{c\over r^{N+p}}\|u(x_0+\cdot)-u(x_0)-l_{\nabla u(x_0)}\big\|^p_{L^p(rY;\RR^m)}
\end{eqnarray*}
with $c>0$ which only depends on $p$. Using \eqref{P-MT-Eq1}, \eqref{P-MT-Eq8} and \eqref{P-MT-Eq7} we deduce that
\begin{equation}\label{P-MT-Eq12}
\lim_{r\to0}\lim_{n\to\infty}\|u_n^r-l_{\nabla u(x_0)}\|_{L^p(Y;\RR^m)}=0.
\end{equation}
According to \eqref{P-MT-Eq12}, \eqref{P-MT-Eq6-Bis} and \eqref{P-MT-Eq10} together with \eqref{P-MT-Eq11}, by diagonalization there exists a mapping $n\to r_n$ decreasing to $0$ such that:
\begin{eqnarray}
&&\left\{
\begin{array}{l}
v_n\to l_{\nabla u(x_0)}\hbox{ in }L^p(Y;\RR^m)\\
\lim\limits_{n\to\infty}\int_{Y}L(\nabla v_n(y))dy<\infty,\hbox{ and }\sup\limits_n\int_YL(\nabla v_n(y))dy<\infty;
\end{array}
\right.\label{Hyp-Cpq}\\
&&\{\nabla v_n\}_n\hbox{ generates }\mu_{x_0}\hbox{ as a Young measure}.\label{Hyp-CT-YM}
\end{eqnarray}
where $v_n:=u_n^{r_n}$. 

\subsection*{Step 3: using (C\boldmath$_{p,q}$)\unboldmath} According to \eqref{Hyp-Cpq}, by (C$_{p,q}$) there exists $\{w_n\}_n\subset l_{\nabla u(x_0)}+W^{1,q}_0(Y;\RR^m)$ such that
$$
\left\{
\begin{array}{l}
|\nabla v_n-\nabla w_n|\to0\hbox{ in measure}\\
L(\nabla w_n)\hbox{ is equi-integrable,}
\end{array}
\right.
$$
hence, by \eqref{Hyp-CT-YM} and  Lemma \ref{Lemma1-YM}, $\{\nabla w_n\}_n$ generates $\mu_{x_0}$ as a Young measure. In particular, $\sup_n\int_YL(\nabla w_n(y))dy<\infty$, and so $\nabla w_n(y)\in\LL$ for all $n\geq 1$ and a.a. $y\in Y$. As $L$ is continuous on $\LL$ and $\LL$ is closed, taking \eqref{P-MT-Eq5} into account, from Theorem \ref{C-T-YM} we deduce that
\begin{equation}\label{P-MT-Eq13}
\lim_{n\to\infty}\int_YL(\nabla w_n(y))dy=\langle L;\mu_{x_0}\rangle.
\end{equation}
On the other hand, by definition of $\Z_qL$, we see that 
$$
\int_YL(\nabla w_n(y))dy\geq \Z_qL(\nabla u(x_0))\hbox{ for all }n\geq 1,
$$
and \eqref{P-MT-Eq5-goal} follows by letting $n\to\infty$ and using \eqref{P-MT-Eq13}. 
\end{proof}

\begin{remark}
Analyzing the step  2 of the proof of (\ref{FI-MTh-1}), it is easily seen that we have also proved the following lemma (that we will used in the proof of Theorem \ref{Main-Theorem-Bis}).
\begin{lemma}\label{blow-up-lemma}
Let $t\in]0,1[$, let $p\in]1,\infty[$, let $G:\MM\to[0,\infty]$ be a Borel measurable and $p$-coercive function, let $u\in W^{1,p}(\Omega;\RR^m)$, let $\{u_n\}_n\subset W^{1,p}(\Omega;\RR^m)$  and let $(\mu_x)_{\in\Omega}$ be a family of probability measures on $\MM$. Assume that
$$
\left\{
\begin{array}{l}
u_n\to u\hbox{ in }L^p(\Omega;\RR^m)\\
\displaystyle\sup_n\int_\Omega G(\nabla u_n(x))dx<\infty\\
\{t\nabla u_n\}_n\hbox{ generates }(\mu_x)_{x\in\Omega} \hbox{ as a Young measure}.
\end{array}
\right.
$$
Then, for a.e. $x_0\in\Omega$, there exists $\{v_n\}_n\subset W^{1,p}(Y;\RR^m)$ such that
$$
\left\{
\begin{array}{l}
v_n\to l_{\nabla u(x_0)}\hbox{ in }L^p(\Omega;\RR^m)\\
\displaystyle\sup_n\int_Y G(\nabla v_n(x))dx<\infty\\
\{t\nabla v_n\}_n\hbox{ generates }\mu_{x_0} \hbox{ as a Young measure}.
\end{array}
\right.
$$
\end{lemma}
\end{remark}

\begin{proof}[\bf Proof of (\ref{FI-MTh-2})]
Given $u\in\Aff(\Omega;\RR^m)$ there exists a finite family $\{U_i\}_{i\in I}$ of open disjoint subsets of $\Omega$ such that $|\Omega\setminus\cup_{i\in I}U_i|=0$ and, for each $i\in I$, $|\partial U_i|=0$ and $\nabla u(x)=\xi_i$ in $U_i$ with $\xi\in\MM$. Thus
\begin{equation}\label{EqUaTiOn-ls-0}
\int_\Omega\Z_qL(\nabla u(x))dx=\sum_{i\in I}|U_i|\Z_qL(\xi_i).
\end{equation}
Recalling that $q\geq p$ and using Lemma \ref{Localization-Lemma-ZG}, for each $i\in I$, we can assert that there exists $\{\phi_n^i\}_n\subset W^{1,p}_0(U_i;\RR^m)$ such that:
\begin{eqnarray}
&&\lim_{n\to\infty}\|\phi_n^i\|_{L^p(U_i;\RR^m)}=0;\label{EqUaTiOn-ls-1}\\
&&\lim_{n\to\infty}\mint_{U_i}L(\xi_i+\nabla\phi_n^i(x))dx=\Z_qL(\xi_i).\label{EqUaTiOn-ls-2}
\end{eqnarray}

Define $\{u_n\}_n\subset W^{1,p}(\Omega;\RR^m)$ by
$$
u_n(x):=u(x)+\phi^i_n(x)\hbox{ if }x\in U_i.
$$
Using \eqref{EqUaTiOn-ls-1} it easy to see that $\|u_n-u\|_{L^p(\Omega;\RR^m)}\to0$, and combining \eqref{EqUaTiOn-ls-2} with \eqref{EqUaTiOn-ls-0} we deduce that
$$
\lim_{n\to\infty}L(\nabla u_n(x))dx=\int_\Omega\Z_qL(\nabla u(x))dx,
$$
and the result follows.
\end{proof}

\begin{remark}
Analyzing the previous proof, it is easily seen that we have in fact proved the following lemma.
\begin{lemma}\label{Z_qG-Second-Lemma}
Let $p\in]1,\infty[$ and $q\in[p,\infty]$ and let $G:\MM\to[0,\infty]$ be a Borel measurable function. For every $u\in\Aff(\Omega;\RR^m)$ there exists $\{u_n\}_n\subset W^{1,p}(\Omega;\RR^m)$ such that
$$
\left\{
\begin{array}{l}
\displaystyle\lim_{n\to\infty}\|u_n-u\|_{L^p(\Omega;\RR^m)}=0\\
\displaystyle\lim_{n\to\infty}\int_\Omega G(\nabla u_n(x))dx=\int_{\Omega}\Z_qG(\nabla u(x))dx.
\end{array}
\right.
$$
\end{lemma}
\end{remark}

\subsection{Proof of Theorem \ref{Main-Theorem}(b)} It is sufficient to prove that
$$
\overline{I}(u)\leq\int_\Omega\Z_qL(\nabla u(x))dx\hbox{ for all }u\in W^{1,p}(\Omega;\RR^m)\setminus\Aff(\Omega;\RR^m).
$$
Let $u\in W^{1,p}(\Omega;\RR^m)\setminus\Aff(\Omega;\RR^m)$ be such that $\int_\Omega\Z_qL(\nabla u(x))dx<\infty$. By (${\rm H}_{p,q}$) there exists $\{u_{k}\}_{k}\subset\Aff(\Omega;\RR^m)$ such that
\begin{equation}\label{EqT-Th1-bIs}
\left\{
\begin{array}{l}
\displaystyle\lim_{k\to\infty}\|u_{k}-u\|_{L^p(\Omega;\RR^m)}=0\\
\displaystyle\limsup_{k\to\infty}\int_\Omega\Z_qL(\nabla u_{k}(x))dx\leq\int_\Omega\Z_qL(\nabla u(x))dx.
\end{array}
\right.
\end{equation}
From Lemma \ref{Z_qG-Second-Lemma} we deduce that for every $k\geq 1$, there exists $\{u_{n,k}\}_n\subset W^{1,p}(\Omega;\RR^m)$ such that
\begin{equation}\label{EqT-Th2-bIs}
\left\{
\begin{array}{l}
\displaystyle\lim_{n\to\infty}\|u_{n,k}-u_{k}\|_{L^p(\Omega;\RR^m)}=0\\
\displaystyle\lim_{n\to\infty}\int_\Omega L(\nabla u_{n,k}(x))dx=\int_{\Omega}\Z_qL(\nabla u_{k}(x))dx.
\end{array}
\right.
\end{equation}
Combining \eqref{EqT-Th2-bIs} with \eqref{EqT-Th1-bIs}, we conclude that
$$
\left\{
\begin{array}{l}
\displaystyle\limsup_{k\to\infty}\lim_{n\to\infty}\|u_{n,k}-u\|_{L^p(\Omega;\RR^m)}=0\\
\displaystyle\limsup_{k\to\infty}\lim_{n\to\infty}\int_\Omega L(\nabla u_{n,k}(x))dx\leq\int_{\Omega}\Z_qL(\nabla u(x))dx.
\end{array}
\right.
$$
and the result follows by diagonalization. $\blacksquare$ 


\section{Proof of Theorem \ref{Main-Theorem-Bis}}

In this section we prove Theorem \ref{Main-Theorem-Bis}.

\subsection{Proof of Theorem \ref{Main-Theorem-Bis}(a)} We are going to prove the following two inequalities: 
\begin{eqnarray}
&&\overline{I}(u)\geq\int_\Omega\widehat{\Z_qL}(\nabla u(x))dx\hbox{ for all }u\in W^{1,p}(\Omega;\RR^m);\label{FI-MT2-1}\\
&&\overline{I}(u)\leq\int_\Omega\widehat{\Z_qL}(\nabla u(x))dx\hbox{ for all }u\in \Aff(\Omega;\RR^m).\label{FI-MT2-2}
\end{eqnarray}

\begin{proof}[\bf Proof of (\ref{FI-MT2-1})]
Consider $u\in W^{1,p}(\Omega;\RR^m)$  and $\{u_n\}_n\subset W^{1,p}(\Omega;\RR^m)$ such that 
\begin{equation}\label{P-MT-Eq1-bis}
\|u_n-u\|_{L^p(\Omega;\RR^m)}\to0,
\end{equation}
and prove that
\begin{equation}\label{aim-eq}
\liminf_{n\to\infty}I(u_n)\geq \int_\Omega\widehat{\Z_qL}(\nabla u(x))dx.
\end{equation}
 Without loss of generality we can assume that
\begin{equation}\label{Eq-Fundamental}
\infty>\liminf_{n\to\infty}I(u_n)=\lim_{n\to\infty}I(u_n) \hbox{ and so }\sup_n\int_\Omega L(\nabla u_n(x))dx<\infty.
\end{equation}
\subsection*{Step 1: using ru-usc assumption, i.e., \boldmath(R$_1$)\unboldmath} According to the definition of $\widehat{\Z_qL}$, see \eqref{Def-Z-hat-i}, and Fatou's lemma, to establish \eqref{aim-eq} it is sufficient to show that 
\begin{equation}\label{P-MT-Eq0-bis}
\liminf_{n\to\infty}I(u_n)\geq \int_\Omega{\Z_qL}(t\nabla u(x))dx.
\end{equation}
for all $t\in]0,1[$. On the other hand, for any $c>0$,
\begin{equation}\label{Eq-Fundamental1}
\int_\Omega L(t\nabla u_n(x))dx\leq(1+\Delta^c_{L}(t))\int_\Omega L(\nabla u_n(x))dx
+c|\Omega|\Delta^c_{L}(t)
\end{equation}
for all $n\geq 1$ and all $t\in]0,1[$, where $\Delta_{L}^c(t)$ is given by \eqref{Ru-usc-Z-hat-i}, and consequently
$$
\liminf_{t\to1}\liminf_{n\to\infty}\int_\Omega L(t\nabla u_n(x))dx\leq\liminf_{n\to\infty}\int_\Omega L(\nabla u_n(x))dx
$$
because $L$ is ru-usc, i.e., $\limsup_{t\to1}\Delta_{L}^c(t)\leq0$ for some $c>0$. Hence, we are reduced to prove that for every $t\in]0,1[$,
\begin{equation}\label{aim-eq-1}
\liminf_{n\to\infty}\int_\Omega L(t\nabla u_n(x))dx\geq \int_\Omega{\Z_qL}(t\nabla u(x))dx.
\end{equation}

\subsection*{Step 2: localization} Fix $t\in]0,1[$. Using \eqref{Eq-Fundamental} and \eqref{Eq-Fundamental1} we see that
\begin{equation}\label{P-MT-Eq2-BIS-bis}
\sup_n\int_\Omega L(t\nabla u_n(x))dx<\infty.
\end{equation}
Since $L$ is $p$-coercive, $\sup_n\int_\Omega|t\nabla u_n(x)|^pdx<\infty$ by \eqref{P-MT-Eq2-BIS-bis}. From Theorem \ref{ExiStenCE-Young-measures} we deduce that there exists a family $(\mu_x)_{x\in\Omega}$ of probability measures on $\MM$ such that (up to a subsequence)
\begin{equation}\label{P-MT-Eq3-bis}
\{t\nabla u_n\}_n\hbox{ generates }(\mu_x)_{x\in\Omega}\hbox{ as a Young measure}.
\end{equation}
On the other hand, \eqref{Eq-Fundamental} implies that for every $n\geq 1$ and a.e. $x\in\Omega$, $\nabla u_n(x)\in\LL$, hence 
\begin{equation}\label{Eq-Blow-up-lemma-1}
t\nabla u_n(x)\in t\overline{\LL}\hbox{ for all }n\geq 1\hbox{ and a.a. }x\in\Omega.
\end{equation}
By (R$_2$) we have $t\overline{\LL}\subset{\rm int}(\LL)$, where ${\rm int}(\LL)$ denotes the interior of $\LL$, and consequently $L$ is continuous on $t\overline{\LL}$ because $L$ is continuous on  ${\rm int}(\LL)$. Noticing that $t\overline{\LL}$ is closed and taking \eqref{Eq-Blow-up-lemma-1} into account, from Theorem \ref{S-T-YM} it follows that
$$
\liminf_{n\to\infty}\int_\Omega L(t\nabla u_n(x))dx\geq \int_{\Omega}\langle L;\mu_x\rangle dx
$$
with (because \eqref{P-MT-Eq2-BIS-bis} holds) for a.e. $x_0\in\Omega$,
\begin{equation}\label{P-MT-Eq5-bis}
\langle L;\mu_{x_0}\rangle<\infty.
\end{equation}
Thus, to prove \eqref{aim-eq-1} it is sufficient to show that for a.e. $x_0\in\Omega$,
\begin{equation}\label{P-MT-Eq5-goal-bis}
\langle L;\mu_{x_0}\rangle\geq \Z_qL(t\nabla u(x_0)).
\end{equation}
\subsection*{Step 3: using Lemma \ref{blow-up-lemma} and (\boldmath$\widehat{\rm C}_{p,q}$\unboldmath)} Taking \eqref{P-MT-Eq1-bis}, \eqref{Eq-Fundamental} and \eqref{P-MT-Eq3-bis} into account, from Lemma \ref{blow-up-lemma} we deduce that for a.e. $x_0\in\Omega$, there exists $\{v_n\}_n\subset W^{1,p}(Y;\RR^m)$ such that:
\begin{equation}\label{blow-up-lemma-eq}
\left\{
\begin{array}{l}
v_n\to l_{\nabla u(x_0)}\hbox{ in }L^p(\Omega;\RR^m)\\
\displaystyle\sup_n\int_Y L(\nabla v_n(x))dx<\infty;
\end{array}
\right.
\end{equation}
\begin{equation}\label{blow-up-lemma-eq-1}
\{t\nabla v_n\}_n\hbox{ generates }\mu_{x_0} \hbox{ as a Young measure}.
\end{equation}
Fix any $x_0\in\Omega$ such that \eqref{blow-up-lemma-eq} and \eqref{blow-up-lemma-eq-1} hold. By $(\widehat{\rm C}_{p,q})$ there exist $\{w_n\}_n\subset l_{t\nabla u(x_0)}+W^{1,q}_0(Y;\RR^m)$ and $s\in]0,1[$ such that
$$
\left\{
\begin{array}{l}
|t\nabla v_n-\nabla w_n|\to0\hbox{ in measure}\\
L(\nabla w_n)\hbox{ is equi-integrable}\\
\nabla w_n(y)\in s\overline{\LL}\hbox{ for all }n\geq 1\hbox{ and a.a. }y\in Y,
\end{array}
\right.
$$
hence, by \eqref{blow-up-lemma-eq-1} and  Lemma \ref{Lemma1-YM}, $\{\nabla w_n\}_n$ generates $\mu_{x_0}$ as a Young measure. As $L$ is continuous on $s\overline{\LL}$ and $s\overline{\LL}$ is closed, taking \eqref{P-MT-Eq5-bis} into account, from Theorem \ref{C-T-YM} we deduce that
\begin{equation}\label{P-MT-Eq13-bis}
\lim_{n\to\infty}\int_YL(\nabla w_n(y))dy=\langle L;\mu_{x_0}\rangle.
\end{equation}
On the other hand, by definition of $\Z_qL$, we see that 
$$
\int_YL(\nabla w_n(y))dy\geq \Z_qL(t\nabla u(x_0))\hbox{ for all }n\geq 1,
$$
and \eqref{P-MT-Eq5-goal-bis} follows by letting $n\to\infty$ and using \eqref{P-MT-Eq13-bis}. 
\end{proof}

\begin{proof}[\bf Proof of (\ref{FI-MT2-2})]
Let $u\in\Aff(\Omega;\RR^m)$ be such that $\int_\Omega\widehat{\Z_qL}(\nabla u(x))dx<\infty$. Then 
\begin{equation}\label{EQT-1}
\nabla u(x)\in\widehat{\Z_q\LL}\hbox{ for all }n\geq 1\hbox{ and a.a. }x\in\Omega,
\end{equation}
where $\widehat{\Z_q\LL}$ denotes the effective domain of $\widehat{\Z_qL}$. Since $L$ is ru-usc, also is $\Z_qL$. Moreover, by (R$_3$),
\begin{equation}\label{EQT-2}
t\overline{\Z_q\LL}\subset{\rm int}(\Z_q\LL)\hbox{ for all }t\in]0,1[,
\end{equation}
where $\overline{\Z_q\LL}$ (resp. ${\rm int}(\Z_q\LL)$) denotes the closure (resp. the interior) of $\Z_q\LL$, and $\Z_qL$ is continuous on ${\rm int}(\Z_q\LL)$ by Lemma \ref{Fonseca-Lemma}. From Theorem \ref{Extension-Result-for-ru-usc-Functions} it follows that:
\begin{eqnarray}
&&
\widehat{\Z_qL}(\xi)=\left\{
\begin{array}{ll}
\Z_qL(\xi)&\hbox{if }\xi\in{\rm int}(\Z_q\LL)\\
\lim\limits_{t\to 1}\Z_qL(t\xi)&\hbox{if }\xi\in\partial (\Z_q\LL)\\
\infty&\hbox{otherwise{\rm;}}
\end{array}
\right.\label{EQT-3}
\\
&&\widehat{\Z_qL}\hbox{ is ru-usc, i.e.,} \limsup_{t\to1}\Delta_{\widehat{\Z_qL}}^c(t)\leq 0\hbox{ for some }c>0.\label{EQT-4}
 \end{eqnarray}
 By \eqref{EQT-3} we have $\widehat{\Z_q\LL}\subset\overline{\Z_q\LL}$, and so $t\nabla u(x)\in{\rm int}(\Z_q\LL)$ for all $t\in]0,1[$ because of \eqref{EQT-1} and \eqref{EQT-2}. Thus
 $$
 \int_\Omega\Z_qL(t\nabla u(x))dx\leq \big(1+\Delta_{\widehat{\Z_qL}}^c(t)\big)\int_\Omega\widehat{\Z_qL}(\nabla u(x))dx+c|\Omega|\Delta_{\widehat{\Z_qL}}^c(t)
 $$
 for all $t\in]0,1[$, and consequently
 \begin{equation}\label{EQT-5}
 \limsup_{t\to1}\int_\Omega\Z_qL(t\nabla u(x))dx\leq\int_\Omega\widehat{\Z_qL}(\nabla u(x))dx
 \end{equation}
 because \eqref{EQT-4} holds. On the other hand, it is clear that
\begin{equation}\label{EQT-bis-1}
\lim_{t\to 1}\|tu-u\|_{L^p(\Omega;\RR^m)}=0.
\end{equation} 
As $u\in\Aff(\Omega;\RR^m)$ we have  $tu\in\Aff(\Omega;\RR^m)$ for all $t\in]0,1[$. From Lemma \ref{Z_qG-Second-Lemma} we deduce that for each $t\in]0,1[$, there exists $\{u_{n,t}\}_n\subset W^{1,p}(\Omega;\RR^m)$ such that
$$
\left\{
\begin{array}{l}
\displaystyle\lim_{n\to\infty}\|u_{n,t}-tu\|_{L^p(\Omega;\RR^m)}=0\\
\displaystyle\lim_{n\to\infty}\int_\Omega L(\nabla u_{n,t}(x))dx=\int_{\Omega}\Z_qL(t\nabla u(x))dx,
\end{array}
\right.
$$
and so
$$
\left\{
\begin{array}{l}
\displaystyle\lim_{t\to 1}\lim_{n\to\infty}\|u_{n,t}-u\|_{L^p(\Omega;\RR^m)}=0\\
\displaystyle\limsup_{t\to1}\lim_{n\to\infty}\int_\Omega L(\nabla u_{n,t}(x))dx\leq\int_\Omega\widehat{\Z_qL}(\nabla u(x))dx
 \end{array}
 \right.
 $$
 by using \eqref{EQT-bis-1} and \eqref{EQT-5}, and the result follows  by diagonalization.
\end{proof}

\subsection{Proof of Theorem \ref{Main-Theorem-Bis}(b)} It is sufficient to prove that
$$
\overline{I}(u)\leq\int_\Omega\widehat{\Z_qL}(\nabla u(x))dx\hbox{ for all }u\in W^{1,p}(\Omega;\RR^m)\setminus\Aff(\Omega;\RR^m).
$$
Let $u\in W^{1,p}(\Omega;\RR^m)\setminus\Aff(\Omega;\RR^m)$ be such that $\int_\Omega\widehat{\Z_qL}(\nabla u(x))dx<\infty$. Arguing as in the proof of the inequality \eqref{FI-MT2-2}, we have \eqref{EQT-bis-1} and  \eqref{EQT-5} and for every $t\in]0,1[$,
$$
\left\{
\begin{array}{l}
\displaystyle\int_\Omega\Z_qL(t\nabla u(x))dx<\infty\\
t\nabla u(x)\in t\overline{\Z_q\LL} \hbox{ for a.a. }x\in\Omega.
\end{array}
\right.
$$
Fix any $t\in]0,1[$. By ($\widehat{\rm H}_{p,q}$) there exists $\{u_{k,t}\}_{k}\subset\Aff(\Omega;\RR^m)$ such that
\begin{equation}\label{EqT-Th1}
\left\{
\begin{array}{l}
\displaystyle\lim_{k\to\infty}\|u_{k,t}-tu\|_{L^p(\Omega;\RR^m)}=0\\
\displaystyle\limsup_{k\to\infty}\int_\Omega\Z_qL(\nabla u_{k,t}(x))dx\leq\int_\Omega\Z_qL(t\nabla u(x))dx.
\end{array}
\right.
\end{equation}
Fix any $k\geq 1$. By Lemma \ref{Z_qG-Second-Lemma} there exists $\{u_{n,k,t}\}_n\subset W^{1,p}(\Omega;\RR^m)$ such that
\begin{equation}\label{EqT-Th2}
\left\{
\begin{array}{l}
\displaystyle\lim_{n\to\infty}\|u_{n,k,t}-u_{k,t}\|_{L^p(\Omega;\RR^m)}=0\\
\displaystyle\lim_{n\to\infty}\int_\Omega L(\nabla u_{n,k,t}(x))dx=\int_{\Omega}\Z_qL(\nabla u_{k,t}(x))dx.
\end{array}
\right.
\end{equation}
Combining \eqref{EqT-Th2}, \eqref{EqT-Th1} with \eqref{EQT-bis-1} together with \eqref{EQT-5}, we conclude that
$$
\left\{
\begin{array}{l}
\displaystyle\limsup_{t\to1}\limsup_{k\to\infty}\lim_{n\to\infty}\|u_{n,k,t}-u\|_{L^p(\Omega;\RR^m)}=0\\
\displaystyle\limsup_{t\to1}\limsup_{k\to\infty}\lim_{n\to\infty}\int_\Omega L(\nabla u_{n,k,t}(x))dx\leq\int_{\Omega}\widehat{\Z_qL}(\nabla u(x))dx.
\end{array}
\right.
$$
and the result follows by diagonalization. $\blacksquare$ 


\section{Some applications}

In this section we apply our two main results, i.e., Theorems \ref{Main-Theorem} and \ref{Main-Theorem-Bis}, to the relaxation of nonconvex integrals with exponential-growth.

\subsection{General conditions on \boldmath$L$\unboldmath\  for the validity of (C\boldmath$_{p,p}$\unboldmath) and (\boldmath$\widehat{\rm C}_{p,p}$)\unboldmath\  when \boldmath$p\in]N,\infty[$\unboldmath}  (In what follows, $\LL$ denotes the effective domain of $L$.) Let us consider the following two conditions on $L$:
\begin{itemize}
\item[(A$_1$)] there exists $\lambda:]1,\infty[\to]0,1[$ such that $\lambda(R)\to 1$ as $R\to\infty$ and
$$
\lim_{R\to\infty}\sup\left\{{L(\lambda(R)\xi)\over L(\xi)}:\xi\in \LL\hbox{ and }|\xi|\geq R\right\}=0;
$$
\item[(A$_2$)] there exists $\alpha_1>0$ such that 
$$
L(t\xi)\leq \alpha_1(1+L(\xi))
$$
for all $\xi\in\LL$ and all $t\in[0,1]$.
\end{itemize}

The following theorem was proved in \cite[Theorem 1.7]{jpm11}.

\begin{theorem}\label{Appli-Theorem}
 Assume that $L$ is finite, i.e., $\LL=\MM$, and continuous. Assume futhermore that $L$ satisfies  {\rm (A$_1$)}, {\rm (A$_2$)} and 
 \begin{itemize}
\item[(A$_3$)] for every $\xi\in\MM$, there exist $\eta_\xi>0$ and $\alpha_{2,\xi}>0$ such that
$$
L(\xi+s(\zeta-\xi)+a)\leq \alpha_{2,\xi}(1+L(\zeta))
$$
for all $s\in[0,1]$, all $\zeta,a\in\MM$ with $|a|\leq \eta_\xi$.
\end{itemize}
Then {\rm(C$_{p,p}$)} holds for all $p\in]N,\infty[$.
\end{theorem}

The following theorem is a variant of \cite[Theorem 1.7]{jpm11}.

\begin{theorem}\label{Appli-Theorem-Bis}
Assume that $\LL$ is convex, $0\in{\rm int}(\LL)$ and $L$ is continuous on ${\rm int}(\LL)$. Assume futhermore that $L$ satisfies {\rm (A$_1$)}, {\rm (A$_2$)} and 
\begin{itemize}
\item[({$\widehat{\rm A}$}$_3$)] for every $\zeta\in\LL$, there exist $\eta_\zeta>0$ and $\alpha_{2,\zeta}>0$ such that
$$
L(s(\zeta+\hat s(\hat\zeta-\zeta))+(1-s)a)\leq \alpha_{2,\zeta}(1+L(\hat\zeta))
$$
for all $s,\hat s\in[0,1]$, all $\hat\zeta\in\LL$ and all $a\in\MM$ with $|a|\leq \eta_\zeta$.
\end{itemize}
Then {\rm($\widehat{\rm C}_{p,p}$)} holds for all $p\in]N,\infty[$. 
\end{theorem}


\begin{proof}[\bf Proof of Theorem \ref{Appli-Theorem-Bis}]
Let $p\in]N,\infty[$ and let $t\in]0,1[$. Let $\xi\in\MM$ and let $\{v_n\}_n\subset W^{1,p}(Y;\RR^m)$ be such that: 
\begin{eqnarray}
&& v_n\to l_\xi\hbox{ in }L^{p}(Y;\RR^m);\label{Eq-Appli-1-bis}\\
&&\sup_n\int_YL(\nabla v_n(y))dy<\infty, \hbox{ and so} \label{Eq-Appli-3}\\
&&\nabla v_n(y)\in \LL \hbox{ for all }n\geq 1\hbox{ and a.a. }y\in Y \label{Eq-Appli-3-bis}.
\end{eqnarray}
As $L$ is $p$-coercive, from \eqref{Eq-Appli-3} we have
\begin{equation}\label{Eq-Appli-1-bis-bis}
\sup_n\int_Y|\nabla v_n(y)|^pdy<\infty.
\end{equation}
Combining \eqref{Eq-Appli-1-bis} with \eqref{Eq-Appli-1-bis-bis} we deduce that, up to a subsequence,
\begin{equation}\label{Eq-Appli-1}
v_n\wto l_\xi\hbox{ in }W^{1,p}(Y;\RR^m)
\end{equation}
Since $\LL$ is convex, by \eqref{Eq-Appli-1} and \eqref{Eq-Appli-3-bis} we have 
\begin{equation}\label{Eq-Appli-4-1}
\xi\in\overline{\LL}.
\end{equation}
As $p\in]N,\infty[$, \eqref{Eq-Appli-1} implies that, up to a subsequence, 
\begin{equation}\label{Eq-Appli-4}
\|v_n- l_\xi\|_{L^\infty(Y;\RR^m)}\to0.
\end{equation}
\subsection*{Step 1: using the biting Lemma} First, recall Sychev's version of the biting lemma (see \cite[Lemma 3.2]{sychev05a}).
\begin{lemma}\label{biting-lemma}
Let $\{f_n\}_n\subset L^1(Y;[0,\infty[)$ be such that $\sup_n\int_Yf_n(y)dy<\infty$. Then, there exist a subsequence $\{f_n\}_n$ (not relabeled) and $\{M_n\}_n\subset]0,\infty[$ with $M_n\to\infty$ such that $\{f_n\chi_{Y_n}\}_n$ is equi-integrable with $\chi_{Y_n}$ denoting the characteristic function of $Y_n:=\{y\in Y:f_n(y)\leq M_n\}$.
\end{lemma}

Fix any $s\in]0,1[$. Taking \eqref{Eq-Appli-3} and \eqref{Eq-Appli-3-bis}, from (A$_2$) we see that
\begin{equation}\label{Eq-Appli-3-bis-bis}
\sup_n\int_YL(s\nabla v_n(y))dy<\infty
\end{equation}
According to \eqref{Eq-Appli-3-bis-bis}, from Lemma \ref{biting-lemma}, that we apply with $f_n=L(s\nabla v_n)$, we can assert that, up to a subsequence, 
\begin{equation}\label{Eq-Appli-5}
\{L(s\nabla v_n)\chi_{Y_n}\}_n\hbox{ is equi-integrable.} 
\end{equation}
Let $\{R_n\}_n$ be given by
$
R_n:=\essinf_{y\in Y\setminus Y_n}|s\nabla v_n(y)|.
$
By \eqref{Eq-Appli-3-bis} we have $s\nabla v_n(y)\in s\overline{\LL}$  for a.a. $y\in Y$.  Moreover, $s\overline{\LL}\subset{\rm int}(\LL)$ because $\LL$ is convex and $0\in{\rm int}(\LL)$. Thus, since $L$ is continuous on ${\rm int}(\LL)$, $L$ is continuous on $s\overline{\LL}$. Noticing that $s\overline{\LL}$ is closed, $Y\setminus Y_n=\{y\in Y:L(s\nabla v_n(y))>M_n\}$ and $M_n\to\infty$, we have
$
R_n\to\infty.
$
Let $\{u_n\}_n\subset W^{1,p}(Y;\RR^m)$ be defined by
$$
u_n:=\lambda_nv_n\hbox{ with }\lambda_n:=\lambda(R_n),
$$
where $\lambda:]1,\infty[\to]0,1[$, with $\lambda(R_n)\to1$, is given by (A$_{1}$). As $\lambda\overline{\LL}\subset\LL$ for all $\lambda\in]0,1[$, taking \eqref{Eq-Appli-3-bis} into account, it follows that
\begin{equation}\label{Eq-Appli-4-1-1}
\nabla u_n(y)\in\LL\hbox{ for all }n\geq 1\hbox{ and a.a. }y\in Y.
\end{equation}
From \eqref{Eq-Appli-1-bis-bis} and \eqref{Eq-Appli-4} we have:
\begin{eqnarray}
&& \|\nabla v_n-\nabla u_n\|_{L^p(Y;\MM^{m\times N})}\to0;\label{Eq-Appli-6bis}\\
&&\|u_n-l_\xi\|_{L^\infty(Y;\RR^m)}\to0.\label{Eq-Appli-6}
\end{eqnarray}
On the other hand, given any $n\geq 1$, $L(s\nabla u_n)=L(\lambda_n s\nabla v_n)\chi_{Y_n}+L(\lambda_n s\nabla v_n)\chi_{Y\setminus Y_n}$, and so $L(s\nabla u_n)\leq\alpha_1(1+L(s\nabla v_n)\chi_{Y_n})+L(\lambda_n s\nabla v_n)\chi_{Y\setminus Y_n}$ by using (A$_{2}$) together with \eqref{Eq-Appli-3-bis}. Moreover $|s\nabla v_n|\geq R_n$ on $Y\setminus Y_n$, hence 
$$
L(s\nabla u_n)\leq\alpha_1\left(1+L(s\nabla v_n)\chi_{Y_n}\right)+\sup\left\{{L(\lambda_n\zeta)\over L(\zeta)}:\zeta\in \LL\hbox{ and }|\zeta|\geq R_n\right\}L(s\nabla v_n).
$$
Taking \eqref{Eq-Appli-3} and \eqref{Eq-Appli-5} into account and noticing that 
$$
\sup\left\{{L(\lambda_n\zeta)\over L(\zeta)}:\zeta\in \LL\hbox{ and }|\zeta|\geq R_n\right\}\to0
$$
by (A$_1$), we see that $\{L(s\nabla u_n)\}_n$ is equi-integrable. By using (A$_2$) we conclude that
\begin{equation}\label{Eq-Appli-7}
\{L(\hat s\nabla u_n)\}_n\hbox{ is equi-integrable for all }\hat s\in]0,s].
\end{equation}

\subsection*{Step 2: cut-off method} For each $n\geq 1$, we consider $\phi_n\in C_c^\infty(Y;[0,1])$ a cut-off function between $Q_n:=]{\eps_n-1\over2},{1-\eps_n\over 2}[^N$ and $Y$ such that $\|\phi_n\|_{L^\infty(Y)}\leq{2\over\eps_n}$ with 
$$
\eps_n:=\sqrt{{\|u_n-l_\xi\|}}_{L^\infty(Y;\RR^m)}.
$$
(Note that, by \eqref{Eq-Appli-6}, $\eps_n\to0$.) Define $w_n\in l_{t\xi}+W^{1,p}_0(Y;\RR^m)$ by
$$
w_n:=t(l_\xi+\phi_n(u_n-l_\xi))=l_{t\xi}+\phi_n(tu_n-l_{t\xi}).
$$
Then
$$
\nabla w_n=\left\{
\begin{array}{ll}
t\nabla u_n&\hbox{on }Q_n\\
t\left(\xi+\phi_n(\nabla u_n-\xi)\right)+(1-\sqrt{t})\left({t\over 1-\sqrt{t}}\nabla\phi_n\otimes(u_n-l_\xi)\right)&\hbox{on }Y\setminus Q_n.
\end{array}
\right.
$$
Since $\LL$ is convex, also are $\overline{\LL}$ and $\sqrt{t}\ \overline{\LL}$. As $\xi\in\overline{\LL}$ by \eqref{Eq-Appli-4-1}, $0\leq \phi_n\leq 1$, $\nabla u_n\in\LL\subset\overline{\LL}$ by \eqref{Eq-Appli-4-1-1}, we have $\xi+\phi_n(\nabla u_n-\xi)\in\overline{\LL}$, and so 
$
\sqrt{t}(\xi+\phi_n(\nabla u_n-\xi))\in \sqrt{t}\ \overline{\LL}.
$ 
Since $0\in{\rm int}(\LL)$, $B_\eps(0)\subset\LL\subset\overline{\LL}$ for some $\eps>0$, where $B_\eps(0):=\{a\in\MM:|a|<\eps\}$. As $|{(\sqrt{t}/ (1-\sqrt{t}))}\nabla\phi_n\otimes(u_n-l_\xi)|\leq {(2\sqrt{t}/(1-\sqrt{t}))}\eps_n\to0$, we can assert that ${(\sqrt{t}/ (1-\sqrt{t}))}\nabla\phi_n\otimes(u_n-l_\xi)\in B_\eps(0)$, and so ${(t/ (1-\sqrt{t}))}\nabla\phi_n\otimes(u_n-l_\xi)\in \sqrt{t}\ \overline{\LL}$. It follows that
$$
\nabla w_n(y)\in \sqrt{t}\ \overline{\LL}\hbox{ for all }n\geq 1\hbox{ and a.a. }y\in Y.
$$
On the other hand, we have
\begin{equation}\label{Eq-Appli-8}
L(\nabla w_n)\leq L(t\nabla u_n)+L(\nabla w_n)\chi_{Y\setminus Q_n}.
\end{equation}
Noticing that $|{(t/ (1-\sqrt{t}))}\nabla\phi_n\otimes(u_n-l_\xi)|\leq {(2t/ (1-\sqrt{t}))}\eps_n\to0$ and, since $\LL$ is convex, $0\in{\rm int}(\LL)$ and $\xi\in\overline{\LL}$ by \eqref{Eq-Appli-4-1}, $\sqrt{t}\xi\in \sqrt{t}\ \overline{\LL}\subset{\rm int}(\LL)\subset \LL$, from ($\widehat{\rm A}_{3}$) we see that
\begin{equation}\label{Eq-Appli-10}
L(\nabla w_n)\chi_{Y\setminus Q_n}\leq\alpha_{2,\sqrt{t}\xi}(1+L(\sqrt{t}\nabla u_n)).
\end{equation}
Combining \eqref{Eq-Appli-10} with \eqref{Eq-Appli-8} we deduce that for every $n\geq 1$,
$$
L(\nabla w_n)\leq L(t\nabla u_n)+\alpha_{2,\sqrt{t}\xi}(1+L(\sqrt{t}\nabla u_n)).
$$
Taking \eqref{Eq-Appli-7} into account with $s=\sqrt{t}$, we deduce that 
$$
\{L(\nabla w_n)\}_n\hbox{ is equi-integrable.}
$$
Finally, it is easy to see that there exists $K>0$, which only depends on $p$, such that
$$
|\nabla w_n-t\nabla u_n|^p\leq K\left(|\xi|^p+|\nabla v_n|^p\right)\chi_{Y\setminus Q_n} +K\eps_n^p.
$$
Taking \eqref{Eq-Appli-1-bis-bis} into account and recalling that $|Y\setminus Q_n|\to 0$ and $\eps_n\to0$, we deduce that $\|\nabla w_n-t\nabla u_n\|_{L^p(Y;\MM^{m\times N})}\to0$, and so $\|t\nabla v_n-\nabla w_n\|_{L^p(Y;\MM^{m\times N})}\to0$ by combining with \eqref{Eq-Appli-6bis}. It follows that
$$
|t\nabla v_n-\nabla w_n|\to0\hbox{ in measure,}
$$
and the proof is complete.
\end{proof}

\subsection{Relaxation of nonconvex integrals with exponential-growth} From now on, we assume that $L$ has exponential-growth, i.e., 
\begin{equation}\label{Exp-Growth}
\beta {\rm \exp}^{\gamma F(\cdot)}\leq L(\cdot)\leq \alpha\left(1+\exp^{F(\cdot)}\right)
\end{equation}
for some $\alpha,\beta>0$, $\gamma\geq 1$ and some Borel measurable function $F:\MM\to[0,\infty]$. Let us consider the following two conditions on $F$:
\begin{itemize}
\item[(a$_1$)] there exists $r\in]0,\infty[$ such that $\liminf\limits_{|\xi|\to\infty}{F(\xi)\over|\xi|^r}\in]0,\infty];$ 
\item[(a$_2$)] there exists $\eps>0$ such that $\sup\limits_{|a|\leq\eps}F(a)<\infty$. 
\end{itemize}
In what follows, $\FF$ denotes the effective domain of $F$. It is easily seen that, under \eqref{Exp-Growth}, the effective domain of $L$ is equal to that of $F$, i.e., $\LL=\FF$. The following proposition is a slight variant of \cite[Proposition 5.1]{jpm11}.
\begin{proposition}\label{Prop-Appli-1}
If $F$ is finite and if $F$ satisfies {\rm (a$_1$)}, {\rm (a$_2$)} and 
\begin{itemize}
\item[(a$_3$)] there exists $\theta\in]0,\infty[$ such that $F(t\xi)\leq t^\theta F(\xi)$ for all $\xi\in\MM$ and all $t\in[0,1];$ 
\item[(a$_4$)] $F(\zeta-\xi)\leq \sqrt[3]{\gamma}\big(F(\zeta)+F(\xi)\big)$ for all $\zeta,\xi\in\MM;$
\item[(a$_5$)] $F(\zeta+\xi)\leq \sqrt[3]{\gamma}\big(F(\zeta)+F(\xi)\big)$ for all $\zeta,\xi\in\MM$,
\end{itemize}
then $L$ satisfies {\rm (A$_1$)}, {\rm (A$_2$)} and {\rm (A$_3$)}. In particular, from Theorem {\em\ref{Appli-Theorem}} we deduce that {\rm(C$_{p,p}$)} holds for all $p\in]N,\infty[$ whenever $L$ is continuous.
\end{proposition}

Proposition \ref{Prop-Appli-1} can be applied for some convex functions $F$: for example, $F(\cdot)=|\cdot|^\nu$ with $\nu\in[1,\infty[$ and $\gamma=2^{3\nu}$. Proposition \ref{Prop-Appli-1} can be also  applied for some nonconvex functions $F$: for example, $F$ can be of the form $F(\cdot)=f(|\cdot|)$ with $f:[0,\infty[\to[0,\infty[$ an increasing concave function and $\gamma=1$ (see \cite[Corollary 1.8]{jpm11} for more details).

\begin{proof}[\bf Proof of Proposition \ref{Prop-Appli-1}]
First of all, as $L$ has exponential-growth, given any $\lambda:]1,\infty[\to]0,1[$, by using (a$_3$) we have 
$$
{L(\lambda(R)\xi)\over L(\xi)}\leq{\alpha\left(1+\exp^{F(\lambda(R)\xi)}\right)\over\beta\exp^{\gamma F(\xi)}}\leq{\alpha\over\beta}\left(\exp^{-\gamma F(\xi)}+\exp^{\left(\lambda^\theta(R)-\gamma\right) F(\xi)}\right)
$$
for all $R>1$ and all $\xi\in\MM$ with $\alpha,\beta>0$, $\gamma\geq 1$ and $\theta\in]0,\infty[$ given by (a$_3$). From (a$_1$), setting either $\delta:={1\over 2}\liminf_{|\xi|\to\infty}{F(\xi)\over|\xi|^r}$ if $\liminf_{|\xi|\to\infty}{F(\xi)\over|\xi|^r}\in]0,\infty[$ or $\delta=1$ if $\liminf_{|\xi|\to\infty}{F(\xi)\over|\xi|^r}=\infty$ (and so $\lim_{|\xi|\to\infty}{F(\xi)\over|\xi|^r}=\infty$), with $r\in]0,\infty[$,  
we can assert that there exists $R_\delta>1$ such that 
\begin{equation}\label{HypOtHeSis-d2}
F(\xi)\geq \delta|\xi|^r\hbox{ for all }|\xi|\geq R_\delta.
\end{equation}
For each $R\geq R_\delta$, as $\lambda^\theta(R)-\gamma\leq \lambda^\theta(R)-1<0$, by using \eqref{HypOtHeSis-d2}, we see that
$$
{L(\lambda(R)\xi)\over L(\xi)}\leq{\alpha\over\beta}\left(\exp^{-\gamma\delta R^r}+\exp^{\delta\left(\lambda^\theta(R)-1\right) R^r}\right)
$$ 
whenever $|\xi|\geq R$, and  consequently 
$$
\sup_{|\xi|\geq R}{L(\lambda(R)\xi)\over L(\xi)}\leq{\alpha\over\beta}\left(\exp^{-\gamma\delta R^r}+\exp^{\delta\left(\lambda^\theta(R)-1\right) R^r}\right)
$$
for all $R\geq R_\delta$, which shows that (A$_1$) holds with $\lambda(R)=\big(1-R^{-{r\over 2}}\big)^{1\over\theta}$.

By the fact that $L$ has exponential-growth, for every $\xi\in\MM$ and every $t\in[0,1]$,  by using (a$_3$) we have
$$
L(t\xi)\leq \alpha(1+\exp^{F(\xi)})\leq \alpha(1+\exp^{\gamma F(\xi)})\leq\alpha\max\left\{1,{1\over\beta}\right\}(1+L(\xi)),
$$
which shows that (A$_2$) holds with $\alpha_1=\alpha\max\big\{1,{1\over\beta}\big\}$.

Finally, as $L$ has exponential-growth, by using (a$_3$), (a$_4$) and (a$_5$) we see that
$$
L(\xi+t(\zeta-\xi)+a)\leq\alpha\left(1+\exp^{\sqrt[3]{\gamma} F(a)}\exp^{(\gamma+(\sqrt[3]{\gamma})^2) F(\xi)}\exp^{\gamma F(\zeta)}\right)
$$
for all $\xi,\zeta, a\in\MM$ and all $t\in[0,1]$. From (a$_2$) there exists $\eps>0$ such that
$$
M:=\sup_{|a|\leq \eps}F(a)<\infty,
$$
and consequently, for every $\xi\in\MM$, we have
$$
L(\xi+t(\zeta-\xi)+a)\leq\alpha\max\left\{1,{\exp^{\sqrt[3]{\gamma} M}\exp^{(\gamma+(\sqrt[3]{\gamma})^2)F(\xi)}\over\beta}\right\}\left(1+L(\zeta)\right)
$$
for all $t\in[0,1]$, all $\zeta\in\MM$ and all $a\in\MM$ with $|a|\leq \eps$, which shows that (A$_3$) holds with $\alpha_{2,\xi}=\alpha\max\Big\{1,{\exp^{\sqrt[3]{\gamma} M}\exp^{(\gamma+(\sqrt[3]{\gamma})^2)F(\xi)}\over\beta}\Big\}$.
\end{proof}

\medskip

The following proposition is another variant of \cite[Theorem 1.7]{jpm11}.

\begin{proposition}\label{Prop-Appli-2}
If $F$ is convex and $F(0)=0$ and if $F$ satisfies {\rm (a$_1$)} and {\rm (a$_2$)}, then $L$ satisfies {\rm (A$_1$)}, {\rm (A$_2$)} and {\rm ({$\widehat{\rm A}$}$_3$)}. In particular, from Theorem {\rm\ref{Appli-Theorem-Bis}} we deduce that {\rm({$\widehat{\rm C}$}$_{p,p}$)} holds for all $p\in]N,\infty[$ whenever $L$ is continuous on ${\rm int}(\FF)$.
\end{proposition}
\begin{proof}[\bf Proof of Proposition \ref{Prop-Appli-2}]
First of all, as $L$ has exponential-growth, $F$ is convex and $F(0)=0$, given any $\lambda:]1,\infty[\to]0,1[$, we have 
$$
{L(\lambda(R)\xi)\over L(\xi)}\leq{\alpha\left(1+\exp^{F(\lambda(R)\xi)}\right)\over\beta\exp^{\gamma F(\xi)}}\leq{\alpha\over\beta}\left(\exp^{-\gamma F(\xi)}+\exp^{\left(\lambda(R)-\gamma\right) F(\xi)}\right)
$$
for all $R>1$ and all $\xi\in\FF$ with $\alpha,\beta>0$ and some $\gamma\geq 1$. From (a$_1$), setting either $\delta:={1\over 2}\liminf_{|\xi|\to\infty}{F(\xi)\over|\xi|^r}$ if $\liminf_{|\xi|\to\infty}{F(\xi)\over|\xi|^r}\in]0,\infty[$ or $\delta=1$ if $\liminf_{|\xi|\to\infty}{F(\xi)\over|\xi|^r}=\infty$ (and so $\lim_{|\xi|\to\infty}{F(\xi)\over|\xi|^r}=\infty$), with $r\in]0,\infty[$,  
we can assert that there exists $R_\delta>1$ such that 
\begin{equation}\label{HypOtHeSis-d2}
F(\xi)\geq \delta|\xi|^r\hbox{ for all }|\xi|\geq R_\delta.
\end{equation}
For each $R\geq R_\delta$, as $\lambda(R)-\gamma\leq\lambda(R)-1<0$, by using \eqref{HypOtHeSis-d2}, we see that
$$
{L(\lambda(R)\xi)\over L(\xi)}\leq{\alpha\over\beta}\left(\exp^{-\gamma\delta R^r}+\exp^{\delta\left(\lambda(R)-1\right) R^r}\right)
$$ 
whenever $|\xi|\geq R$ and $\xi\in\FF$, and  consequently 
$$
\sup\left\{{L(\lambda(R)\xi)\over L(\xi)}:\xi\in\FF\hbox{ and }|\xi|\geq R\right\}\leq{\alpha\over\beta}\left(\exp^{-\gamma\delta R^r}+\exp^{\delta\left(\lambda(R)-1\right) R^r}\right)
$$
for all $R\geq R_\delta$, which shows that (A$_1$) holds with $\lambda(R)=1-R^{-{r\over 2}}$.

Using again the fact that $L$ has exponential-growth, $F$ is convex and $F(0)=0$, for every $\xi\in\FF$ and every $t\in[0,1]$,  we have
$$
L(t\xi)\leq \alpha(1+\exp^{F(\xi)})\leq\alpha(1+\exp^{\gamma F(\xi)})\leq\alpha\max\left\{1,{1\over\beta}\right\}(1+L(\xi)),
$$
which shows that (A$_2$) holds with $\alpha_1=\alpha\max\big\{1,{1\over\beta}\big\}$.

Finally, from (a$_2$) there exists $\eps>0$ such that
\begin{equation}\label{widehat-A-3-Assumption}
M:=\sup_{|a|\leq \eps}F(a)<\infty.
\end{equation}
As $L$ has exponential-growth and $F$ is convex, we see that
$$
L(s(\zeta+\hat s(\hat \zeta-\zeta))+(1-s)a)\leq\alpha\left(1+\exp^{F(a)}\exp^{F(\zeta)}\exp^{\gamma F(\hat \zeta)}\right)
$$
for all $\zeta,\hat\zeta, a\in\MM$ and all $s,\hat s\in[0,1]$. Consequently, taking \eqref{widehat-A-3-Assumption} into account, for every $\zeta\in\LL$, we have
$$
L(s(\zeta+\hat s(\hat \zeta-\zeta))+(1-s)a)\leq\alpha\max\left\{1,{\exp^{M}\exp^{F(\zeta)}\over\beta}\right\}\left(1+L(\hat\zeta)\right)
$$
for all $s,\hat s\in[0,1]$, all $\hat\zeta\in\LL$ and all $a\in\MM$ with $|a|\leq \eps$, which shows that ({$\widehat{\rm A}$}$_3$) holds with $\alpha_{2,\zeta}=\alpha\max\big\{1,{\exp^{M}\exp^{\gamma F(\zeta)}\over\beta}\big\}$.
\end{proof}

\medskip

The following result is, for (i), a consequence of  Proposition \ref{Prop-Appli-1}, Theorem \ref{Appli-Theorem} and Theorem \ref{Main-Theorem} and, for (ii), a consequence of Proposition \ref{Prop-Appli-2}, Theorem \ref{Appli-Theorem-Bis} and Theorem \ref{Main-Theorem-Bis}. (In what follows, $\Omega$ is a bounded open and strongly star-shaped subset of $\RR^N$, see \S 3.4.)

\begin{corollary}\label{Applications-Of-Main-Theorems}
Consider $p\in]N,\infty[$ and assume that $F$ satisfies {\rm(a$_1$)} and {\rm(a$_2$)}.
\begin{itemize}
\item[(i)] Under {\rm(a$_3$)}, {\rm(a$_4$)} and {\rm(a$_5$)}, if $F$ is finite and if $L$ is continuous, then \eqref{Eq-Th-1} holds for all $u\in\Aff(\Omega;\RR^m)$. If moreover $F$ is convex then \eqref{Eq-Th-1} is satisfied for all $u\in W^{1,p}(\Omega;\RR^m)$.
\item[(ii)] If $F$ is convex and $F(0)=0$ and if $L$ is ru-usc and continuous on ${\rm int}(\FF)$, then \eqref{Eq-Th2} holds for all $u\in W^{1,p}(\Omega;\RR^m)$. 
\end{itemize}
\end{corollary}
\begin{proof}[\bf Proof of Corollary \ref{Applications-Of-Main-Theorems}]
(i) From Proposition \ref{Prop-Appli-1} and Theorem \ref{Appli-Theorem} we deduce that (C$_{p,p}$) holds, and the first part of (i) follows by using Theorem \ref{Main-Theorem}(a). If $F$ is convex, also is $H(\cdot):=\exp^{\gamma F(\cdot)}$. Then, as $\gamma\geq 1$, $L$ has $H$-convex-growth, i.e., $\beta H(\cdot)\leq L(\cdot)\leq \alpha(1+H(\cdot))$. As $F$ is finite and $L$ is continuous, Lemma \ref{ApproX-Lemma-1} shows that (H$_{p,p}$) holds, and the second part of (i) follows by using Theorem \ref{Main-Theorem}(b).

(ii)  From Proposition \ref{Prop-Appli-1} and Theorem \ref{Appli-Theorem}  we deduce that (${\widehat{\rm C}}_{p,p}$) is satisfied. Since $F$ is convex and $\gamma\geq 1$, $L$ has $H$-convex-growth. Hence $\Z_p\LL=\LL=\FF$, which means that  $\Z_p\LL$ and $\LL$ are convex, and so (R$_2$) and (R$_3$) are satisfied. Moreover, as $L$ in continuous on ${\rm int}(\FF)$, Lemma \ref{ApproX-Lemma-1} shows that (${\widehat{\rm H}}_{p,p}$) holds, and (ii) follows by using Theorem \ref{Main-Theorem-Bis}. \end{proof}

\subsection*{Acknowledgments} I gratefully acknowledges the many comments of M. A. Sychev during the preparation of this paper, and for the lectures on ``Young measures and weak convergence theory" that he gave at the University of N\^imes during may-june 2011.

The author also wishes to thank the ``R\'egion Languedoc Roussillon" for financial support through its program ``d'accueil de personnalit\'es \'etrang\`eres" which allowed to welcome M. A. Sychev, from the Sobolev Institute for Mathematics in Russia, in the Laboratory MIPA of the University of N\^imes.

\end{document}